\documentclass[12pt]{article}
\usepackage{amsfonts}
\usepackage{amsmath,amsthm,amssymb}
\usepackage[all,dvips]{xy}
\usepackage{eucal}
\usepackage{comment}
\usepackage{tikz-cd}
\usepackage{geometry}

\newcommand{\Q}{\mathbb{Q}}

\newtheorem{theorem}{Theorem}[section]

\newtheorem{corollary}[theorem]{Corollary}
\newtheorem{conjecture}[theorem]{Conjecture}
\newtheorem{proposition}[theorem]{Proposition}
\newtheorem*{theorem*}{Theorem}
\newtheorem*{corollary*}{Corollary}
\theoremstyle{definition}
\newtheorem{defn}[theorem]{Definition}

\newtheorem{remark}[theorem]{Remark}

\numberwithin{equation}{section}

\author{Houari Benammar Ammar}
\AtEndDocument{\bigskip{\footnotesize%
  \textsc{Département de Mathématiques, UQAM, C.P. 8888, Succursale Centre-ville, Montréal 
(Québec), H3C 3P8, Canada} \par  
  \textit{E-mail address :} \texttt{benammar\_ammar.houari@courrier.uqam.ca} 
}}
\large
\title{Remarks on the Direct Image Sheaf of Logarithmic Pluricanonical Bundles and the Non-Vanishing Conjecture}
\begin{document}
\normalsize
\maketitle
\begin{abstract}
By applying the Chen-Jiang decomposition, we prove that the non-vanishing conjecture holds for an lc pair \((X, \Delta)\), where \(X\) is an irregular variety, provided it holds for lower-dimensional varieties. In the second part, we extend the Catanese-Fujita-Kawamata decomposition to the klt case \((X, \Delta)\), which leads to the existence of sections of \(K_X + \Delta\) in certain situations.
\end{abstract}
\section{Introduction}
One crucial step in Mori’s MMP program is to prove the non-vanishing conjecture; however, it is the first step towards proving the abundance conjecture. 
\begin{conjecture}\label{nonvanishingconjecture}
    Let $(X,\Delta)$ be a lc pair. If $D \sim_{\Q} m(K_{X} + \Delta)$ is Cartier  pseudo-effective, then $\kappa(D) \geq 0$.
\end{conjecture}
\par Thanks to recent advances in techniques for Generic Vanishing Theory, developed by many authors \cite{Haconreine}, \cite{popapareschi1},   \cite{popasurvey}, \cite{popapareschi} and originally introduced in \cite{Greenlazarsfeld} and \cite{simpson}, we have gained significant insights into irregular varieties. These advances have also deepened our understanding of the structure of the pushforward of logarithmic pluricanonical bundles under a map from an irregular variety to an abelian variety. 
\par The following theorem and corollary are deduced in \cite[Lemma 3.5]{Hu}, with the author’s proof based on \cite[Theorem 4.1]{BirkarChen}. In their work, Birkar and Chen require the pair to be relatively of log general type to run the MMP, as shown in \cite{BCHM}, and to prove the abundance conjecture. They also rely on the stronger generic vanishing result from \cite{Haconreine} to establish their proof. In this note, a simple and short proof is provided.

\begin{theorem}\label{maintheorem1}
Let \((X, \Delta)\) be a klt pair such that \(D \sim_{\mathbb{Q}} m(K_X + \Delta)\) is a Cartier pseudo-effective divisor, with \(q(X) > 0\). If the non-vanishing conjecture \ref{nonvanishingconjecture} holds for lower-dimensional klt pairs, then it also holds for \((X, \Delta)\).
\end{theorem}
\par By an easy application of certain forms of canonical bundle formula \cite{haconxu} due to \cite{fujinomori},  we obtain the following corollary.
\begin{corollary}\label{maintheorem2}
   Let \((X, \Delta)\) be an lc pair such that \(D \sim_{\mathbb{Q}} m(K_X + \Delta)\) is a Cartier pseudo-effective divisor, with \(q(X) > 0\). If Conjecture \ref{nonvanishingconjecture} holds for lower-dimensional varieties, then it also holds for \((X, \Delta)\).
\end{corollary}
\par A crucial step in proving the results announced above is the application of the so-called Chen-Jiang decomposition. More precisely, given a morphism \( f: X \to A \), where \( A \) is an abelian variety, the following decomposition is provided by \cite{chenjiang} and \cite{lombardipopaschnell}:
\[
f_{*}(\omega_{X}^{\otimes m}) = \bigoplus_{i \in I}(\alpha_{i} \otimes p_{i}^{*}\mathcal{F}_{i}) \text{.}
\]
\par The morphisms \( p_{i} : A \to A_{i} \) are algebraic fiber spaces, where \( A_{i} \) are abelian varieties, \( \mathcal{F}_{i} \) are nonzero M-regular coherent sheaves on \( A_{i} \), and \( \alpha_{i} \in \operatorname{Pic}^{0}(A) \) are torsion line bundles. Later, in \cite{jiang} and \cite{fanjunmeng}, the Chen-Jiang decomposition is generalized to a klt pair $(X, \Delta)$, they proved that 
if $D \sim_{\Q} m(K_{X} + \Delta)$ is Cartier, Then for every positive integer $N$ such that $f_{*}(ND) \neq 0$, we have 
\[
f_{*}(ND) = \bigoplus_{i \in I}(\alpha_{i} \otimes p_{i}^{*}\mathcal{F}_{i}) \text{.}
\]
In the same articles \cite{jiang} and \cite{fanjunmeng}, the authors asked whether the previous decomposition is still satisfied for an lc pair. Here, we remark that using a canonical bundle formula from \cite{haconxu}, we can find a subsheaf that admits the Chen-Jiang decomposition. More precisely, we have the following proposition.
\begin{proposition}
    \label{maintheorem3}
Let \((X, \Delta)\) be an lc pair such that \(D \sim_{\mathbb{Q}} m(K_X + \Delta)\) is Cartier, and let \(f: X \to A\) be a morphism to an abelian variety. If \(\kappa(D_{|_F}) \geq 0\), where \(F\) is the general fiber of \(f\), then for every positive integer \(N\) that is sufficiently large and divisible such that \(f_*(ND) \neq 0\), there exists a torsion-free subsheaf \(\mathcal{F}\) of \(f_*(ND)\) such that \(\mathcal{F}\) admits the Chen-Jiang decomposition.
\end{proposition}
\par From the above, we observe that MMP problems are more manageable for irregular varieties due to the extensive development of techniques in this setting. The most challenging aspect, however, lies in working with varieties that have no irregularity, as we cannot use morphisms to abelian varieties. In such cases, it is necessary to explore alternatives, using the so-called Catanese-Fujita-Kawamata decomposition particularly relevant from our perspective.
\par We know by \cite{catanese2}, \cite{catanese3}, \cite{fujitapaper}, \cite{haconpopaschnell}, \cite{schnelllombardi} that if $f:X \to Y$ is a surjective morphism with $X$ and $Y$ being smooth varieties, then $f_{*}(mK_{X/Y})$ is a torsion free sheaf, it has a singular metric with semi-positive curvature, satisfies the minimal extention property \cite[Definition 2.1]{schnelllombardi}, and admits the Catanese-Fujita-Kawamata decomposition, that is 
\[
f_{*}(mK_{X/Y}) = \mathcal{A}_{m} \oplus \mathcal{U}_{m} \text{,}
\]
where $A_{m}$ is a generically ample sheaf and $\mathcal{U}_{m}$ is flat. This decomposition holds in the singular case, that is, for \((X, \Delta)\) being klt, and it should well known to experts. Since we could not find it  in the literature, we provide it here.
\begin{theorem}\label{maintheorem5}
     Let \( f: X \to Y \) be a surjective morphism, and let \( (X, \Delta) \) be a klt pair such that \( D \sim_{\mathbb{Q}} m(K_{X/Y} + \Delta) \) is Cartier. Then, for every positive integer \( N \) that is sufficiently large and divisible such that \( f_{*}(ND) \neq 0 \), the sheaf \( f_{*}(ND) \) is torsion-free, it has a singular metric with semi-positive curvature, satisfies the minimal extension property, and admits the Catanese-Fujita-Kawamata decomposition
    \[
    f_{*}(ND) = \mathcal{A}_{N} \oplus \mathcal{U}_{N} \text{.}
    \]
\end{theorem}
\par Klt polarized pairs are important for the minimal model program, and thus we have the following easy corollary.
\begin{corollary}\label{maintheorem6}
    Let $f: X \to Y$ be a surjective morphism, and let $(X, \Delta + L)$ be a klt polarized pair such that $D \sim_{\Q} m(K_{X/Y} + \Delta + L)$ is Cartier and $f$-big. Then for every positive integer $N$ which is sufficiently big and divisible such that $f_{*}(ND) \neq 0$, the sheaf $f_{*}(ND)$ is torsion free, it has a singular metric with semi-positive curvature, satisfies the minimal extention property, and admits the Catanese-Fujita-Kawamata decomposition.
 \[
    f_{*}(ND) = \mathcal{A}_{N} \oplus \mathcal{U}_{N} \text{.}
    \]
    
\end{corollary}
\begin{theorem}\label{maintheorem7}
    Let \( f: X \to Y \) be a flat algebraic fiber space of relative dimension \( p \), and let \( (X, \Delta) \) be a klt pair such that \( D \sim_{\mathbb{Q}} m(K_{X/Y} + \Delta) \) is Cartier. Let \( N \) be a positive integer that is sufficiently large and divisible such that \( f_{*}(ND) \neq 0 \), and assume that \( h^{p}(F, (1 - Nm)K_{F} - Nm\Delta_F) \neq 0 \) is constant for every fiber \( F \), with \( P_N := h^{0}(\mathcal{R}^{p}f_{*}((1 - Nm)K_{X/Y} - Nm\Delta)) > 0 \). Then \( \mathcal{O}_{Y}^{\oplus P_N} \) is a direct summand of \( f_{*}(ND) \).
\end{theorem}
This last theorem shows that we can produce sections for \( K_{X/Y} + \Delta \). Moreover, under an additional positivity condition for \( Y \), we can produce a section for \( K_{X} + \Delta \). This leads to the following corollary.
\begin{corollary}\label{maintheorem8}
Assume the same assumptions of Theorem \ref{maintheorem7}, and $\kappa(Y) \geq 0$. Then $\kappa(K_{X} + \Delta) \geq 0$. Furthermore, if $P_{N} > 1$ then $\kappa(K_{X} + \Delta) \geq 1$.
\end{corollary}
Recall that if \( f: X \to A \) is a morphism to an Abelian variety, then for a klt pair \( (X, \Delta) \), we have the Chen-Jiang decomposition of \( f_{*}(Nm(K_{X} + \Delta)) \) such that $Nm(K_{X} + \Delta) \sim ND$. It is natural to ask the same question for the Catanese-Fujita-Kawamata decomposition under any surjective morphism \( f: X \to Y \). Of course, clearly, we cannot make such a change of base since it is known that \( f_{*}(Nm(K_{X/Y} + \Delta)) \) is not necessarily semi-ample, and the flat part depends on the monodromy group. However, we still have something like the following proposition.
\begin{proposition}\label{maintheorem9}
Let \( f: X \to Y \) be a surjective morphism with \( q(Y) \geq 1 \), and let \( (X, \Delta) \) be a klt pair such that \( D \sim_{\mathbb{Q}} m(K_{X} + \Delta) \) is Cartier. Then, for every positive integer \( N \) that is sufficiently large and divisible such that \( f_{*}(ND) \neq 0 \), there exist finite maps \( p: \tilde{X} \to X \), \( q: \tilde{Y} \to Y \), and \( \phi: \tilde{A} \to \operatorname{Alb}(Y) \) with a Cartier divisor \( \tilde{D} = p^{*}D \), and a surjective morphism \( \tilde{f}: \tilde{X} \to \tilde{Y} \) such that \( (g \circ \tilde{f})_{*}(N \tilde{D}) \) is globally generated, where \( g: \tilde{Y} \to \tilde{A} \).
\end{proposition}
\par In Section 5, we revisit some ideas introduced by Viehweg and explore how they can be used to algebraically derive the existence of the Catanese-Fujita-Kawamata decomposition for a klt pair, assuming we already know the result for the smooth case. For instance, we have Theorem \ref{finitecov1} and Theorem \ref{finitecover2}, which are key observations in this context.
\par \textbf{Notation and conventions.} 
Fix $f: X \to Y$ a proper  morphism of normal projective varieties. We say that $X$ is $\Q$-factorial if every Weil divisor is $\Q$-Cartier. We say that a $\Q$-divisor $D$ is $\Q$-Cartier if some integral multiple is Cartier. We say that two $\Q$-divisors $D_{1}, D_{2}$ on $X$ are $\Q$-linearly equivalent (over $Y$), that is $D_{1} \thicksim_{\Q} D_{2}$ ($D_{1} \thicksim_{\Q, f} D_{2}$) if their difference is an $\Q$-linear combination of principal divisors (and an $\Q$-Cartier divisor pulled back from $Y$). $D_{1}$ and $D_{2}$ are numerically equivalent (over $Y$), denoted $D_{1} \equiv D_{2}$ ($D_{1} \equiv_{f} D_{2}$), if their difference is an $\Q$-Cartier divisor such that $(D_{1} - D_{2}).C =0$ for any curve $C$ (contracted by $f$). A $\Q$-Cartier divisor $D$ is semi-ample over $Y$, or $f$-semi-ample, if  $D$ is $Q$-linearly equivalent to the
pullback of an ample $Q$-divisor over $Y$. Equivalently, $f^{*}f_{*}\mathcal{O}_{X}(mD) \to \mathcal{O}_{X}(mD)$ is surjective for $m>>0$. we say that $\Q$-divisor $D$ is big over $Y$, or $f$-big, if
    $$
    \text{lim sup}_{m \to \infty} \frac{h^{0}(F, \mathcal{O}_{F}(\llcorner mD\lrcorner))}{m^{\dim F}} > 0
    $$
    for the fibre $F$ over any generic point of $Y$. Equivalently $D$ is $f$-big if $D \thicksim_{\Q, f} M + E$ , where $M$ is ample and $E$ effective. We define the Kodaira dimension of a $\Q$-divisor $\Q$-Cartier by 
    $$
    \kappa(D):= \kappa(mD)$$
    for some natural number $m$ such that $mD$ is Cartier.
\par By a pair $(X,\Delta)$, we mean a normal variety $X$ associated with $\Q$-divisor $\Delta: = \sum a_{i}\Delta_{i}$, which is a formal sum of distinct prime divisors $\Delta_{i}$ with $ a_{i} \in [0,1]$ such that $K_{X} + \Delta$ is $\Q$-cartier. For the definitions of singularities of pairs, we refer to \cite{kollarmori}.
\par By a klt polarized pair $(X, \Delta + L)$, we mean a klt pair $(X, \Delta)$ and $L$ is a nef $\Q$-divisor (see \cite[Paragraph 2.2]{Birkarandhu}). 
\par We note the Albanese dimension $\alpha(X)$ of an irregular variety $X$ by 
    \[
    \alpha(X):= \dim \operatorname{alb}(X).
    \]
    Here $\operatorname{alb}(X)$ is the image of the Albanese map $\operatorname{alb}: X \to \operatorname{alb}(X) \subseteq \operatorname{Alb}(X)$, where $\operatorname{Alb}(X)$ is the Albanese variety.
\section{Preliminaries}

\par \par We recall the Fourier-Mukai setting, we refer to Mukai \cite{Mukai} for more details. We denote by $\mathcal{P}$ the Poincaré line bundle on $A \times \operatorname{Pic}^{0}(A)$, and $A$ is an abelian variety of dimension $g$. For any coherent sheaf $\mathcal{F}$ on $A$, we can associate the sheaf ${p_{2}}_{*}(p_{1}^{*}\mathcal{F} \otimes \mathcal{P})$ on $\operatorname{Pic}^{0}(A)$ where
$p_{1}$ and $p_{2}$ are the natural projections on $A$ and $\operatorname{Pic}^{0}(A)$, respectively. This correspondence gives a
functor
\[
\Hat{S}: \operatorname{Coh}(A) \to \operatorname{Coh}(\operatorname{Pic}^{0}(A)).
\]
If we denote by $\operatorname{D}(A)$ and $\operatorname{D}(\operatorname{Pic}^{0}(A))$ the bounded derived categories of $\operatorname{Coh}(A)$ and
$\operatorname{Coh}(\operatorname{Pic}^{0}(A))$, then the derived functor $$\mathcal{R}\Hat{S}: \operatorname{D}(A) \to \operatorname{D}(\operatorname{Pic}^{0}(A))$$  is defined and called the
Fourier-Mukai functor. Similarly, we consider $$\mathcal{R}S : \operatorname{D}(\operatorname{Pic}^{0}(A))  \to \operatorname{D}(A)$$ in a similar
way. According to the celebrated result of Mukai \cite{Mukai}, the Fourier-Mukai functor induces an equivalence of categories between the two derived categories  $\operatorname{D}(A)$ and $\operatorname{D}(\operatorname{Pic}^{0}(A))$. More precisely, we have
\[
\mathcal{R}S \circ \mathcal{R}\Hat{S} \simeq (-1_{A})^{*}[-g]
\]
and 
\[
\mathcal{R}\Hat{S} \circ \mathcal{R}S \simeq (-1_{\operatorname{Pic}^{0}(A)})^{*}[-g]\text{,}
\]
where $[-g]$ is a shift operation for a complex $g$ places to the right.
\par \par We recall some basic definitions and properties of sheaves on abelian varieties, we refer to  \cite{Greenlazarsfeld},
\cite{popasurvey} and \cite{popapareschi} for more details.
\begin{defn}[\cite{Greenlazarsfeld}]
Let $\mathcal{F}$ be a coherent sheaf on an abelian variety $A$.  The set $V^{i}(\mathcal{F})$ is \emph{the cohomology support loci} and defined as the following:
    \[
    V^{i}(\mathcal{F}) := \{p \in \operatorname{Pic}^{0}(A)| H^{i}(A, \mathcal{F} \otimes p) \neq 0\}
    \text{.}\]
\end{defn}
Recall that, in \cite{Greenlazarsfeld}, the authors proved a powerful generic vanishing theorem for $\omega_{Y}$ with $Y$ being an irregular variety. In other words, a generic vanishing theorem for $\operatorname{alb}_{*}(\omega_{Y})$. 
  \begin{defn}[{\cite[Definition 3.1]{popasurvey}}] A coherent sheaf $\mathcal{F}$ on $Y$ is called \emph{M-regular} if 
 \[
\text{codim}_{\operatorname{Pic}^{0}(Y)}\text{(Supp}(\mathcal{R}^{i}\Hat{S}(\mathcal{F}))) > i, \hspace{0.2cm} \forall i \geq 1.
\]
    \par Or equivalentely 
    \[
    \text{codim}_{\operatorname{Pic}^{0}(Y)}(V^{i}(\mathcal{F})) > i, \hspace{0.2cm} \forall i \geq 1.
    \]
A coherent sheaf $\mathcal{F}$ on $Y$ is called a \emph{generic vanishing sheaf} or a \emph{GV-sheaf} if its cohomology support loci $V^{i}(\mathcal{F})$ satisfy the following inequality: 
 \[
    \text{codim}_{\operatorname{Pic}^{0}(Y)}(V^{i}(\mathcal{F})) \geq i, \hspace{0.2cm} \forall i \geq 1.
    \]
\end{defn}
\par We recall certain forms of the canonical bundle formula, originally introduced by Fujino and Mori (\cite{fujinomori}).
\begin{theorem}\label{canonicalformula}[{\cite[Theorem 2.1]{haconxu}, \cite[Theorem 3.1]{Hu}}] 
Let \( f : (X, \Delta) \to Y \) be a projective morphism from an lc pair to a normal variety \( Y \), such that \( N(K_X + \Delta) \) is Cartier and \( f_*(N(K_X + \Delta)) \neq 0 \) for some integer \( N > 0 \). Then, there exists a commutative diagram:
\[
\begin{tikzcd}
\Tilde{X} \arrow[r, "\psi"] \arrow[d, "\Tilde{f}"] & X \arrow[d, "f"] \\
\Tilde{Y} \arrow[r, "\phi"] & Y
\end{tikzcd}
\]
with the following properties:
\begin{itemize}
    \item [(1)] $\psi$ is a birational morphism, $\tilde{X}$ is smooth and $\tilde{f}$ is an algebraic fiber space. 
    \item[(2)] There exist a $\Q$-divisor $\Tilde{\Delta}$ such that $(\Tilde{X}, \Tilde{\Delta})$ is a $\Q$-factorial  dlt pair, and
    \[
 \psi_{*}(N(K_{\Tilde{X}} + \Tilde{\Delta})) = N(K_{X} + \Delta).
    \]
    
    \item [(3)] There exist a $\Q$-factorial dlt polarized pair $(\Tilde{Y}, \Delta_{\Tilde{Y}} + L_{\Tilde{Y}} )$ pair such that $K_{\Tilde{Y}}+ \Delta_{\Tilde{Y}} + L_{\Tilde{Y}}$ is big $/Y$.
    \item[(4)] There exist an effective $\Q$- divisor $R$ on $\Tilde{X}$ such that $\Tilde{f}_{*}(\mathcal{O}_{\Tilde{X}}(mR)) = \mathcal{O}_{\Tilde{Y}}$ for all $m \geq 0$, 
    \[
    K_{\Tilde{X}} + \Tilde{\Delta} \sim_{\Q} \Tilde{f}(K_{\Tilde{Y}}+ \Delta_{\Tilde{Y}} + L_{\Tilde{Y}}) + R \text{,}\]
    and 
    \[
    \tilde{f}_{*}(N(K_{\tilde{X}} + \tilde{\Delta})) = N(K_{\tilde{Y}}+ \Delta_{\tilde{Y}} + L_{\tilde{Y}}).
    \]
    \item [(5)] each component of $\lfloor \Delta_{\Tilde{Y}} \rfloor$ is dominated by a vertical component of $\lfloor\Tilde{\Delta}\rfloor$. 
    \end{itemize}
\end{theorem}
\section{Non-vanishing and Chen-Jiang decomposition}
In the introduction, we highlight that the Albanese map and certain generic vanishing techniques can be used to produce sections of log pluricanonical bundles for irregular varieties. The key element is the application of the Chen-Jiang decomposition.
\begin{proof}[Proof of Theorem \ref{maintheorem1}]
We consider the Albanese morphism $\operatorname{alb}: X \to \operatorname{Alb}(X)$. If $\alpha(X) = n$, then the morphism $\operatorname{alb}$ is generically finite, and thus $\operatorname{alb}_{*}(D) \neq 0$. If $\alpha(X) < n$, we take the Stein factorization $f: X \to Y$ of $\operatorname{alb}$ and denote its general fiber by $F$. Clearly, the lower-dimensional pair $(F, \Delta_{|_{F}})$ is klt, and $Nm(K_{F} + \Delta|_{F})$ has a section for every positive integer $N$ which is sufficiently big and divisible by hypothesis. Thus, $f_{*}(Nm(K_{X} + \Delta)) \neq 0$, which implies $\operatorname{alb}_{*}(Nm(K_{X} + \Delta)) \neq 0$. By the decomposition results of \cite{jiang} and \cite{fanjunmeng}, we have 
\[
\operatorname{alb}_{*}(ND) = \bigoplus_{i \in I}(\alpha_{i} \otimes p_{i}^{*}\mathcal{F}_{i}) 
\]
The morphisms \( p_{i} : \operatorname{Alb}(X) \to A_{i} \) are algebraic fiber spaces, where \( A_{i} \) are abelian varieties, \( \mathcal{F}_{i} \) are nonzero M-regular coherent sheaves on \( A_{i} \), and \( \alpha_{i} \in \operatorname{Pic}^{0}(\operatorname{Alb}(X)) \) are torsion line bundles of finite orders. Choose any $\alpha_{j}$ in the decomposition, then 
\[
\operatorname{alb}_{*}(ND) \otimes \alpha_{j}^{-1} = p_{j}^{*}\mathcal{F}_{j} \bigoplus_{i \in I-\{j\}}(\alpha_{i} \otimes \alpha_{j}^{-1} \otimes p_{i}^{*}\mathcal{F}_{i}) \text{.}
\]
It is clear that any \(M\)-regular sheaf has a section. Indeed, by \cite{popapareschi1}, we know that any \(M\)-regular sheaf is continuously globally generated, which implies \(h^{0}(A_{j}, \mathcal{F}_{j} \otimes \alpha) \neq 0\) for general \(\alpha \in \operatorname{Pic}^{0}(\operatorname{Alb}(X))\). By semi-continuity, we then have \(h^{0}(A_{j}, \mathcal{F}_{j}) \neq 0\).  

Now, since \(p_{j}\) is an algebraic fiber space, it follows from the projection formula that \(h^{0}(\operatorname{Alb}(X), p_{j}^{*}\mathcal{F}_{j}) \neq 0\). Thus,  
\[
h^{0}(\operatorname{Alb}(X), \operatorname{alb}_{*}(ND) \otimes \alpha_{j}^{-1}) \neq 0,
\]  
which implies \(h^{0}(X, ND \otimes p_{j}^{*}\alpha_{j}^{-1}) \neq 0\).  

Assume that the order of \(\alpha_{j}^{-1}\) is \(k\). Then, \(h^{0}(X, kND) \neq 0\), which completes the proof of the theorem.  
\end{proof}
\begin{remark}
If \( \Delta = 0 \), then by \cite{Caopaun}, we know that the \( C_{n,m} \) conjecture is true for an algebraic fiber space over a variety with maximal Albanese dimension, and of course, we can deduce Theorem \ref{maintheorem1}. Also, for a klt pair \( (X, \Delta) \), the \( C_{n,m} \) conjecture for the same algebraic fiber space is satisfied by the work of Birkar and Chen \cite{BirkarChen}, but the proof involves many reduction steps, and we should use some technical extension theorems as given in \cite{demaillyhaconchen}.
\end{remark}
\begin{proof}[Proof of Corollary \ref{maintheorem2}]
We apply Theorem \ref{canonicalformula} to obtain the following commutatif diagram
\[
\begin{tikzcd}
\Tilde{X} \arrow[r, "\psi"] \arrow[d, "\Tilde{f}"] & X \arrow[d, "f= \operatorname{alb}"] \\
\Tilde{Y} \arrow[r, "\phi"] & Y = \operatorname{Alb}(X)
\end{tikzcd}
\]
such that the properties $(1), \dots, (5)$ are satisfied. Note by $P$ the vertical component of $\lfloor \tilde{\Delta}\rfloor$, since  each component of $\lfloor \Delta_{\Tilde{Y}} \rfloor$ is dominated by  $P$, it is clear that we can find a klt polarized pair $(\tilde{Y}, \Delta^{'} + L^{'})$ and a $\Q$-Cartier divisor $R^{'}$ such that $K_{\tilde{Y}} + \Delta^{'} + L^{'}$ is big $/Y$ and for some sufficiently small $\epsilon$ 
\[
K_{\Tilde{X}} + \Tilde{\Delta} -\epsilon P \sim_{\Q} \Tilde{f}(K_{\tilde{Y}} + \Delta^{'} + L^{'}) + R^{'} \text{,}
\]
and 
    \[
    \tilde{f}_{*}(Nm(K_{\tilde{X}} + \tilde{\Delta} - \epsilon P)) = Nm(K_{\tilde{Y}} + \Delta^{'} + L^{'}).
    \]
    (Without loss of generality, assume $Nm(K_{\tilde{X}} + \tilde{\Delta} - \epsilon P)$ and $Nm(K_{\tilde{Y}} + \Delta^{'} + L^{'})$ are Cartier).
    since $K_{\tilde{Y}} + \Delta^{'} + L^{'}$ is big $/Y$, we have
\[
K_{\tilde{Y}} + \Delta^{'} + L^{'} \sim_{\Q, \phi} M + E,
\]
where $M$ is an ample $\Q$-divisor on $\tilde{Y}$, and $E$ is effective. Then, for some $\delta > 0$, we can find $\Delta_\delta$ such that $(\tilde{Y}, \Delta_\delta)$ is klt and
\begin{equation}\label{eq0}
K_{\tilde{Y}} + \Delta_{\delta} \sim_{\mathbb{Q}, \phi} K_{\tilde{Y}} + \Delta^{'} + L^{'} + \delta E  + \delta M \sim_{\mathbb{Q}, \phi} (1 + \delta)(K_{\tilde{Y}} + \Delta^{'} + L^{'}).
\end{equation}
Here $K_{\tilde{Y}} + \Delta_{\delta}$ is big $/Y$, and for some $N$ which is sufficiently big and divisible, $Nm(K_{\tilde{Y}} + \Delta_{\delta})$ is Cartier. By Theorem \ref{maintheorem1}, $\kappa(K_{\tilde{Y}} + \Delta_{\delta}) \geq 0$, hence $\kappa(K_{\tilde{Y}} + \Delta^{'} + L^{'}) \geq 0$, thus  the sheaf $\tilde{f}_{*}(Nm(K_{\tilde{X}} + \tilde{\Delta} - \epsilon P))$ has a nonzero section, which implies that $Nm(K_{\Tilde{X}} + \Tilde{\Delta} -\epsilon P)$ has a section. Finally $Nm(K_{\Tilde{X}} + \Tilde{\Delta})$ and $Nm(K_X + \Delta)$ have a section.
\end{proof}
\begin{remark}
We do not know if the Chen-Jiang decomposition holds for a lc pair \((X, \Delta)\) since we miss a semi-positivity result for the pushforward of the log pluricanonical bundle. Otherwise, Corollary \ref{maintheorem2} would follow automatically without the use of any form of the canonical bundle formula. As we mentioned in the introduction, in \cite{jiang} and \cite{fanjunmeng}, the authors asked if the Chen-Jiang decomposition is satisfied for the pushfoward of a lc pairs. We remark that, by using the canonical bundle formula, we can see that \( f_{*}(Nm(K_{X} + \Delta)) \) contains a subsheaf that admits the Chen-Jiang decomposition for every positive integer \(N\) that is sufficiently large and divisible such that \(f_*(ND) \neq 0\).
 \end{remark}
\begin{proof}[Proof of Proposition \ref{maintheorem3}]
By assumption, \(\kappa(K_{F} + \Delta|_{F}) \geq 0\). Thus, we apply Theorem \ref{canonicalformula} to the pair $(X, \Delta)$, obtaining the following diagram:
\[
\begin{tikzcd}
\Tilde{X} \arrow[r, "\psi"] \arrow[d, "\Tilde{f}"] & X \arrow[d, "f"] \\
\Tilde{Y} \arrow[r, "\phi"] & Y=A
\end{tikzcd}
\]
such that the properties \((1), \dots, (5)\) are satisfied. By following the same steps as in the proof of Corollary \ref{maintheorem2}, we find that for some \(\delta > 0\), there exists \(\Delta_\delta\) such that \((\tilde{Y}, \Delta_\delta)\) is klt. Without loss of generality, we assume that the divisors \(Nm(K_{\tilde{Y}} + \Delta_{\delta})\), \(Nm(1 + \delta)(K_{\tilde{Y}} + \Delta' + L')\), and \(Nm(K_{\tilde{X}} + \tilde{\Delta} - \epsilon P)\) are Cartier. It is clear that the following torsion free sheaf \[(\phi \circ \tilde{f})_{*}(Nm(K_{\tilde{X}} +\tilde{\Delta} -\epsilon P)) =\phi_{*}(Nm(K_{\tilde{Y}} + \Delta_{\delta})) \neq 0 \text{,}\]
and it admits the Chen-Jiang decomposition. We know that the diagram above is commutative, thus 
\[
\mathcal{F} := (f \circ \psi)_{*}(Nm(K_{\tilde{X}} + \tilde{\Delta} - \epsilon P)) \neq 0,
\]
which admits the Chen-Jiang decomposition. Note that the sheaf \(\mathcal{F}\) is a torsion-free subsheaf of \(f_{*}(Nm(K_{X} + \Delta))\).
\end{proof}
\section{Catanese-Fujita-Kawamata decomposition}
As mentioned in the introduction, it is not clear how to produce sections for \(K_{X} + \Delta\) even if there exist some for \(K_{F} + \Delta_{|_{F}}\). In the case where the morphism is to an abelian variety, this is known from the discussion above and is also well known to experts. However, if the target space is not an abelian variety, it presents an obstruction. Exploring the Catanese-Fujita-Kawamata decomposition is therefore relevant to our purpose.
\begin{defn}[{\cite[Definition 1.1]{schnelllombardi}}]\label{decompositiondef}
A coherent torsion-free sheaf \(\mathcal{F}\) admits a Catanese–\\Fujita–Kawamata decomposition if it decomposes in the following form
\[
\mathcal{F} \cong \mathcal{U} \oplus \mathcal{A},
\]
where \(\mathcal{U}\) is a Hermitian flat bundle, and \(\mathcal{A}\) is either a generically ample sheaf or the zero sheaf.
\end{defn}

\par Recall the following theorem proven in \cite{schnelllombardi}.
\begin{theorem}[{\cite[Theorem 1.3]{schnelllombardi}}]\label{schnelltheorem}
Let \(\mathcal{F}\) be a coherent torsion-free sheaf on a smooth projective variety \(Y\), endowed with a singular Hermitian metric with semi-positive curvature and satisfying the minimal extension property. Then \(\mathcal{F}\) admits a Catanese–Fujita–Kawamata decomposition.    
\end{theorem}
\par As an example for the previous theorem, in \cite{schnelllombardi} the authors deduced the decomposition theorem (Definition \ref{decompositiondef}) for $f_{*}(m(K_{X/Y}))$ where $f: X \to Y$ is a surjective morphsim between smooth varieties. We remark that the decomposition is satisfied for the klt case. 
\begin{proof}[Proof of Theorem \ref{maintheorem5}]
The proof is classical, we refer to \cite{pauntakayama} for such details. $Nm\Delta$ is Cartier, and by assumption  we have $f_{*}(Nm(K_{X/Y}+\Delta)) \neq 0$, then define the divisor 
\[
D_{N} := (Nm-1)K_{X/Y} + Nm\Delta 
\]
By \cite{pauntakayama}, the divisor above admits a singular hermitian metric with semi-positive curvature, and the following inclusion is generically an isomorphism
\[f_{*} \big( \mathcal{O}_{X}(K_{X/Y} + D_{N}) \otimes \mathcal{I}(h_{N}) \big) \hookrightarrow f_{*} \mathcal{O}_{X}(K_{X/Y} + D_{N})= f_{*}(Nm(K_{X/Y}+ \Delta)),\]
here $h_{N}$ is the associated metric to $D_{N}$ and $\mathcal{I}(h_{N})$ is the associated multiplier ideal sheaf for it. We know by the famous result of \cite{pauntakayama} (see also \cite{berndtsoonpaun}, \cite{haconpopaschnell}, \cite{horing}, \cite{paunsurvey}) 
that $f_{*} \big( \mathcal{O}_{X}(K_{X/Y} + D_{N}) \otimes \mathcal{I}(h_{N}) \big)$ admits a singular hermitian metric with semi-positive and satisfies the minimal extention property. The inclusion above is generically an isomorphism, thus by \cite[Proposition 2.2]{schnelllombardi}, the torsion free sheaf $f_{*} \mathcal{O}_{X}(K_{X/Y} + D_{N})= f_{*}(Nm(K_{X/Y}+ \Delta))$ is endowed with a singular hermitian metric with semi-positive curvature and satisfies the minimal extention property. Then we can apply Theorem \ref{schnelltheorem} to conclude.
\end{proof}
\begin{proof}[Proof of Corollary \ref{maintheorem6}]
By assumption $K_{X/Y} + \Delta + L$ is big $/Y$. Then  we have
\[
K_{X/Y} + \Delta + L \sim_{\Q, f} M + E,
\]
where $M$ is an ample $\Q$-divisor on $Y$, and $E$ is effective. Then, for some $\delta > 0$, we can find $\Delta_\delta$ such that $(X, \Delta_\delta)$ is klt and
\begin{equation}\label{eq0}
K_{X/Y} + \Delta_{\delta} \sim_{\mathbb{Q}, f} K_{X/Y} + \Delta + L + \delta E  + \delta M \sim_{\mathbb{Q}, f} (1 + \delta)(K_{X/Y} + \Delta + L).
\end{equation}
Then we apply Theorem \ref{maintheorem5} to deduce the decomposition.
\end{proof}
\par It is natural to ask whether the flat part or the generically ample part has a section. The author believes that a deeper understanding of the Catanese-Fujita-Kawamata decomposition is crucial for making progress on positivity problems.
\begin{proof}[Proof of Theorem \ref{maintheorem7}]
By assumption, \(h^{p}(F, (1-Nm)K_{F} - Nm\Delta_{F}) \neq 0\) is constant for every fiber \(F\), and the algebraic fiber space is flat. Then, by Grauert's theorem, we deduce that the coherent sheaf \(\mathcal{R}^{p}f_{*}((1-Nm)K_{X/Y} - Nm\Delta)\) is locally free. Note \[P_{N}:= h^{0}(Y, \mathcal{R}^{p}f_{*}((1-Nm)K_{X/Y} -Nm\Delta)) = h^{0}(Y, f_{*}(Nm(K_{X/Y}+ \Delta))^{\vee}),\]  since \[ f_{*}(Nm(K_{X/Y}+ \Delta))^{\vee}\simeq \mathcal{R}^{p}f_{*}((1-Nm)K_{X/Y} -Nm\Delta).\] We take $\{s_{1}, \dots, s_{P_{N}}\}$ as a basis of \[H^{0}(Y, \mathcal{R}^{p}f_{*}((1-Nm)K_{X/Y} -Nm\Delta)) \simeq Hom(\mathcal{O}_{Y}, \mathcal{R}^{p}f_{*}((1-Nm)K_{X/Y} -Nm\Delta)).\] Then $s_{1} \oplus \dots \oplus s_{P_{N}}$ defines a map
\[
s_{1} \oplus \dots \oplus s_{P_{N}}: \mathcal{O}_{Y}^{\oplus P_{N}} \longrightarrow \mathcal{R}^{p}f_{*}((1-Nm)K_{X/Y} -Nm\Delta)
\]
which yields the following short exact sequence
\[
0 \longrightarrow \mathcal{O}_{Y}^{\oplus P_{N}} \longrightarrow \mathcal{R}^{p}f_{*}((1-Nm)K_{X/Y} -Nm\Delta) \longrightarrow \mathcal{Q}_{N} \longrightarrow 0
\]
where $\mathcal{Q}_{N}$ is the quotient sheaf. By duality, we have
\[
0 \longrightarrow \mathcal{Q}_{N}^{\vee} \longrightarrow f_{*}(Nm(K_{X/Y}+ \Delta)) \longrightarrow \mathcal{O}_{Y}^{\oplus P_{N}} \longrightarrow 0 \text{.}
\]
We deduce from \cite[Theorem 26.4]{haconpopaschnell} that the last exact sequence splits since  the bundle $f_{*}(Nm(K_{X/Y}+ \Delta))$ admits a singular hermitian metric with semi-positive curvature and satisfies the minimal extension property
\end{proof}
\begin{proof}[Proof of Corollary \ref{maintheorem8}]
By assumption $\kappa(Y) \geq 0$, then for some positive integer $N$ which is sufficiently big and divisible, we have $NmK_{Y}$ is effective. By Theorem \ref{maintheorem7}, we know that $\mathcal{O}_{Y}^{\oplus P_{N}}$ is a sub-sheaf of $f_{*}(Nm(K_{X/Y}+ \Delta))$. Now, by the following multplication map 
\[
H^{0}(Y, f_{*}(Nm(K_{X/Y}+ \Delta))) \times H^{0}(Y, NmK_{Y}) \to H^{0}(Y, f_{*}(Nm(K_{X}+ \Delta))) \text{,}
\]
we produce sections for the torsion free sheaf $f_{*}(Nm(K_{X}+ \Delta))$, and of course for the divisor $Nm(K_{X}+ \Delta))$. Thus, we deduce that $\kappa(K_{X} + \Delta) \geq 0$. The second assertion is clear.
\end{proof}
\begin{proof}[Proof of Proposition \ref{maintheorem9}]
By hypothesis $q(Y) \geq 1$, thus $(\operatorname{alb} \circ f)_{*}(Nm(K_{X} + \Delta))$ has the Chen-Jiang decomposition if it is not zero, furthermore, we can find an isogeny $\phi: \tilde{A} \to \operatorname{Alb}(Y)$ such that $\phi^{*}((\operatorname{alb} \circ f)_{*}(m(K_{X} + \Delta)))$ is globally generated, and we have the following diagram
\[
\begin{tikzcd}
\Tilde{X} \arrow[r, "p"] \arrow[d, "\Tilde{f}"] & X \arrow[d, "f"] \\\Tilde{Y} \arrow[r, "q"] \arrow[d, "g"] & Y \arrow[d, "\operatorname{alb}"]\\
\Tilde{A} \arrow[r, "\phi"] & \operatorname{Alb}(Y)
\end{tikzcd}
\]
The maps $p: \Tilde{X} \to X$, $q: \Tilde{Y} \to Y$ and $\phi: \tilde{A} \to \operatorname{Alb}(Y)$ are finites. The morphism $\tilde{f} : \tilde{X} \to \tilde{Y}$ is a surjective morphism. Clearly, the pair $(\tilde{X}, \tilde{\Delta})$ is klt such that $K_{\tilde{X}} + \tilde{\Delta} := p^{*}(K_{X} + \Delta)$. Hence we deduce the result from the commutativity of the diagram above.
\end{proof}
\section{Viehweg's trick machinery}
\par In this section, we review some techniques introduced by Viehweg, commonly referred to as Viehweg’s trick machinery \cite{Viehweg} (see also \cite{kawamatacyclic}). He developed these techniques to make significant progress in studying the positivity of the direct image sheaf of the pluricanonical bundle, which has direct applications to the $C_{n,m}$ conjecture. For instance, this conjecture was resolved by Cao and Păun \cite{Caopaun} in the important case where the variety is fibred over an Abelian variety. They reduced the problem, à la Kawamata \cite{kawamataabelianvariety}, to the case where the variety has trivial Kodaira dimension over an Abelian variety. In this setting, they established a crucial positivity result for $f_{*}(mK_{X})$ (\cite{Caopaun}).
\par We are inspired from  \cite{lombardipopaschnell}, \cite{fanjunmeng}, \cite{popaschnell}, \cite{Viehweg} to deduce the followings.
\begin{theorem}\label{finitecov1}
Let $f: X \to Y$ be a surjective morphism, and let $(X, \Delta)$ be a klt pair such that $D \sim_{\Q} (K_{X/Y} + \Delta)$ is Cartier. Then there exist a smooth variety $Z$ with a generically finite map $h: Z \to X$ such that $f_{*}\mathcal{O}(D)$ is a direct summand of $g_{*}\omega_{Z/Y}:= (f\circ h)_{*}\omega_{Z/Y}$.
\end{theorem}
\begin{proof}
   It is enough to assume that $(X, \Delta)$ is a klt log smooth pair. Indeed, take a log resolution of $(X, \Delta)$, $\mu: \tilde{X}  \to X$ such that 
   \[
   K_{\tilde{X}/Y} + \Delta_{\tilde{X}} \sim_{\Q} \mu^{*}(K_{X/Y} + \Delta) +E \text{,}
   \]
   where $\Delta_{\tilde{X}}$ and $E$ are effective SNC and have no common components, $E$ is the exceptional divisor. Therefore, we have
   \[
   K_{\tilde{X}/Y} + \Delta_{\tilde{X}} +\lceil E\rceil - E \sim_{\Q} \mu^{*}(K_{X/Y} + \Delta) + \lceil E\rceil \text{.}
    \]
    We put $\Delta_{\tilde{X}}^{'}: = \Delta_{\tilde{X}} +\lceil E\rceil - E$, then  the pair $(\tilde{X}, \Delta_{\tilde{X}}^{'})$ is klt log smooth.
    \par Furthermore 
    \[
    (f\circ \mu)_{*}(K_{\tilde{X}/Y} + \Delta_{\tilde{X}}^{'}) = (f \circ \mu)_{*}(\mu^{*}D + \lceil E\rceil) =f_{*}(D) \text{.}
    \]
    Thus, we can work with the log smooth klt pair $(\tilde{X}, \Delta_{\tilde{X}}^{'})$, and therefore assume that $(X, \Delta)$ is klt log smooth.
\par Now, we take $N$ such that 
\[
N(D-K_{X/Y}) \sim N\Delta.
\]
Hence, we see that the Cartier divisor \(N(D - K_{X/Y})\) is divisible. Then, we can take a finite map \(h': Z' \to X\) ramified along \(N\Delta\). By resolving the singularities of \(Z'\), we obtain a generically finite map \(h: Z \to X\), with \(Z\) smooth, such that  
\[
h_{*}\omega_{Z} = \mathcal{O}_{X}(K_{X} + (N-1)(D- K_{X/Y})) \otimes \mathcal{O}_{X}(- \lfloor (N-1)\Delta \rfloor) \bigoplus \dots 
\]
\[
= \mathcal{O}_{X}((D- K_{X/Y} + K_{X}))\otimes \mathcal{O}_{X}( (N-2)(D-K_{X/Y}))\otimes \mathcal{O}_{X}(- \lfloor (N-1)\Delta \rfloor) \bigoplus \dots 
\]
But 
\[
\mathcal{O}_{X}( (N-2)(D-K_{X/Y}))\otimes \mathcal{O}_{X}(- \lfloor (N-1)\Delta \rfloor) = \mathcal{O}_{X} \text{.}
\]
Thus
\[
h_{*}\omega_{Z}= \mathcal{O}_{X}((D- K_{X/Y} + K_{X})) \bigoplus \dots
\]
Hence
\[
h_{*}\omega_{Z/Y}= \mathcal{O}_{X}(D) \bigoplus \dots
\]
Finally, we have 
\[
g_{*}\omega_{Z/Y}:= (f\circ h)_{*}\omega_{Z/Y} = f_{*}\mathcal{O}(D) \bigoplus  \dots
\]
as required.
\end{proof}
\begin{corollary}
Let $f: X \to Y$ be a surjective morphism, and let $(X, \Delta)$ be a klt pair such that $D \sim_{\Q} (K_{X/Y} + \Delta)$ is Cartier. Then $f_{*}\mathcal{O}_{X}(D)$ admits a Catanese-Fujita-Kawamata decomposition.    
\end{corollary}
\begin{theorem}\label{finitecover2}
Let $f: X \to Y$ be a surjective morphism, and let $(X, \Delta)$ be a klt pair such that $D \sim_{\Q} m(K_{X/Y} + \Delta)$ is Cartier for some $m>1$. If $f_{*}\mathcal{O}_{X}(D)$ is globally generated, then there exist a smooth variety $Z$ with a generically finite map $h: Z \to X$ such that $f_{*}\mathcal{O}(D)$ is a direct summand of $g_{*}\omega_{Z/Y}:= (f\circ h)_{*}\omega_{Z/Y}$.
\end{theorem}
\begin{proof}
We have the following evaluation map 
\[
f^{*}f_{*} \mathcal{O}_{X}(D) \to \mathcal{O}_{X}(D),
\]
and the image is $D \otimes \mathcal{I}$, where $\mathcal{I}$ is the relative base ideal of $D$. We take a log resolution of $(X, \Delta)$ and $I$, $\mu : \tilde{X} \to X$ such that 
\[
   K_{\tilde{X}/Y} + \Delta_{\tilde{X}} \sim_{\Q} \mu^{*}(K_{X/Y} + \Delta) +E \text{,}
   \]
    where $\Delta_{\tilde{X}}$ and $E$ are effective SNC and have no common components, $E$ is the exceptional divisor. We put $D_{\tilde{X}} := \mu^{*}D$, and $k: =f \circ \mu$. Then the image of the following evaluation map 
    \[
    k^{*}k_{*}\mathcal{O}_{X}(D_{\tilde{X}}) \to \mathcal{O}_{X}(D_{\tilde{X}})
    \]
    is of the form $\mathcal{O}_{X}(D_{\tilde{X}} -F)$ for some effective SNC divisor $F$. Hence, we define the new boundary divisor
    \[
    \Delta_{\tilde{X}}^{'}:= \Delta_{\tilde{X}} +\frac{\lceil mE\rceil}{m} - E,
    \]
    and clearly the pair $(\tilde{X}, \Delta_{\tilde{X}}^{'})$ is klt log smooth. Therefore
    \[
   m(K_{\tilde{X}/Y} + \Delta_{\tilde{X}}^{'}) \sim_{\Q} \mu^{*}(m(K_{X/Y} + \Delta)) +\lceil mE \rceil \sim_{\Q} \mu^{*}D + \lceil mE \rceil\text{.}
    \]
    Put $G: = \mu^{*}D + \lceil mE \rceil$. Then $k_{*}\mathcal{O}_{\tilde{X}}(G) = f_{*}\mathcal{O}_{X}(D)$, and $k_{*}\mathcal{O}_{\tilde{X}}(G)$ is globally generated. By above, the image of the following evaluation map 
    \[
    k^{*}k_{*}\mathcal{O}_{\tilde{X}}(G) \to \mathcal{O}_{\tilde{X}}(G)
    \]
    is $\mathcal{O}_{\tilde{X}}(G-F - \lceil mE \rceil)$. We define the divisor $G^{'}: = F + \lceil mE \rceil$, and as it pointed in \cite{fanjunmeng}, we have $k_{*}(\mathcal{O}_{\tilde{X}}(G-G^{"})) = k_{*}(\mathcal{O}_{\tilde{X}}(G))$, for any effective divisor $G^{"} \leq G^{'}$.
    \par Since $k_{*}\mathcal{O}_{\tilde{X}}(G)$ is globally generated, then $\mathcal{O}_{\tilde{X}}(G-G^{'})$ is globally generated, and by Bertini's  theorem we can take an effective divisor $H \sim_{\Q} G-G^{'}$, $H$ and $\Delta_{\tilde{X}}^{'} + G^{'}$ have no comments components, with $H + \Delta_{\tilde{X}}^{'} + G^{'}$ is SNC.
\par Now, the goal is to reduce to the structure as in  Theorem \ref{finitecov1}, we can find a new divisor $T \leq G^{'}$ and a klt pair $(\tilde{X}, M)$ such that \[
G-T \sim_{\Q} K_{\tilde{X}/Y} + M,
\]
$T$ and $M$ are given by 
\[
T:= \lfloor \Delta_{\tilde{X}}^{'} + \frac{m-1}{m}G^{'} \rfloor \text{,}
\]
and 
\[
M:= \frac{m-1}{m}H + \Delta_{\tilde{X}}^{'} + \frac{m-1}{m}G^{'} - T
\]
as proven in \cite{fanjunmeng}. The last step is to find a smooth variety $Z$ and a generically finite map $h: Z \to X$ such that $f_{*}\mathcal{O}(D)$ is a direct summand of $g_{*}\omega_{Z/Y}:= (f\circ h)_{*}\omega_{Z/Y}$. The details are the same of Theorem $\ref{finitecov1}$, so we will leave them to the readers.
\end{proof}
\begin{remark}
In Theorem \ref{finitecover2}, if \( f_{*}\mathcal{O}_{X}(D) \) is not globally generated, we can twist the bundle with a sufficiently ample line bundle \( L \) on \( Y \) to ensure that \( f_{*}\mathcal{O}_{X}(D) \otimes L \) is globally generated on \( Y \). In this case, we can prove a statement similar to Theorem \ref{finitecover2}; that is, we can find a smooth variety \( Z \) and a generically finite map \( Z \to X \) such that \( f_{*}\mathcal{O}_{X}(D) \otimes L \) is a direct summand of \( g_{*}\omega_{Z/Y} := (f \circ h)_{*}\omega_{Z/Y} \). All of these observations provide a way for obtaining the Catanese-Fujita-Kawamata decomposition in the logarithmic case.
\end{remark}
\bibliographystyle{plain}
\bibliography{slope paper}
\end{document}